\documentclass[11pt]{article}
\usepackage[T1]{fontenc}
\usepackage[utf8]{inputenc}
\usepackage{lmodern}
\usepackage{amsmath,amssymb,amsthm,mathtools}
\usepackage{enumitem}
\usepackage{geometry}
\usepackage{microtype}
\usepackage{hyperref}
\hypersetup{colorlinks=true,linkcolor=black,citecolor=black,urlcolor=black}

\newtheorem{theorem}{Theorem}[section]
\newtheorem{proposition}[theorem]{Proposition}
\newtheorem{lemma}[theorem]{Lemma}
\newtheorem{corollary}[theorem]{Corollary}
\theoremstyle{definition}
\newtheorem{definition}[theorem]{Definition}
\newtheorem{example}[theorem]{Example}
\theoremstyle{remark}
\newtheorem{remark}[theorem]{Remark}

\newcommand{\Der}{\operatorname{Der}}
\newcommand{\adSN}{\operatorname{ad}_{SN}}

\title{A Note on the Schouten--Nijenhuis Bracket of Lie--Rinehart Algebras}
\author{Servais Cyr Gats\'e \and Basile Guy Richard Bossoto \and Irmely Gladesh Mabanza Nsiloulou}
\date{}

\begin{document}
\maketitle

\begin{center}
\emph{Dedicated to the memory of Professor Eug\`ene Okassa.}
\end{center}

\begin{abstract}
Let $A$ be a commutative associative unital algebra over a field $K$ of characteristic zero, and let $(G,\rho)$ be a Lie--Rinehart algebra over $A$. We revisit the Gerstenhaber structure carried by the exterior algebra $\bigwedge_A G$ from the point of view of graded derivations. Particular attention is paid to the well-definedness of the Schouten--Nijenhuis extension over the coefficient algebra $A$: the Lie--Rinehart identity is shown to provide exactly the correction terms needed for the bracket to descend to the $A$-balanced exterior algebra. After establishing graded antisymmetry, the Leibniz rule and the graded Jacobi identity without circularity, we study the inner graded derivations $\adSN(P)=[P,\cdot]_{SN}$ and prove
\[
[\adSN(P),\adSN(Q)]=\adSN([P,Q]_{SN}).
\]
Finally, we give a complete proof of the converse construction: Gerstenhaber brackets on $\bigwedge_A G$ are in one-to-one correspondence with Lie--Rinehart structures on $G$.
\end{abstract}


\section{Introduction}
Lie--Rinehart algebras provide an algebraic framework combining a Lie algebra structure with a module structure over a commutative algebra. They are the algebraic counterparts of Lie algebroids \cite{Pradines1966} and occur naturally in differential geometry, Poisson geometry and homological algebra. The basic notion goes back to Herz, Palais and Rinehart \cite{Herz1953,Palais1961,Rinehart1963}; its cohomological background is rooted in \cite{ChevalleyEilenberg1948,HKR1962}. Gerstenhaber algebras were introduced in \cite{Gerstenhaber1963}. The correspondence between Lie--Rinehart structures on an $A$-module $G$ and Gerstenhaber brackets on $\bigwedge_A G$ is due to Huebschmann \cite{Huebschmann1990}; see also \cite{Koszul1985,Kosmann1995,Xu1999}, and \cite{Witten1990} for the related antibracket formalism. Generalizations include hom-Lie--Rinehart algebras \cite{MandalMishra2018} and Lie--Rinehart--Jacobi algebras \cite{Okassa2007,Okassa2008,Okassa2011}; Poisson-type structures on exterior algebras are studied in \cite{GatseLikouka2023}.

Let $A$ be a commutative associative unital $K$-algebra and let $(G,\rho)$ be a Lie--Rinehart algebra over $A$. Its exterior algebra
\[
\mathcal G=\bigwedge_A G
\]
carries a natural Schouten--Nijenhuis bracket extending both the Lie bracket on $G$ and the action of $G$ on $A$ through the anchor. Together with the exterior product this gives a Gerstenhaber algebra.

The existence of this correspondence is not presented here as a new fact. Our purpose is to give a careful and self-contained derivation-based formulation in which the role of the anchor in the passage to the $A$-balanced exterior algebra is made explicit. This point is technically important: the Lie bracket on $G$ is not $A$-bilinear, and therefore a formula written first on decomposable tensors must be shown to be independent of the representatives chosen.

A second objective is to organize the construction so that no Gerstenhaber identity is used implicitly to prove itself. We first construct the Schouten--Nijenhuis bracket and prove graded antisymmetry and the Leibniz rule directly from the defining formulas. The graded Jacobi identity is then deduced from these two properties and from the Lie--Rinehart axioms on generators. Only afterwards do we establish, for homogeneous $P\in\bigwedge_A^pG$, the adjoint commutator relation
\[
[\adSN(P),\adSN(Q)]=\adSN([P,Q]_{SN}),
\]
which is interpreted as the adjoint representation of the shifted graded Lie algebra rather than as a route to Jacobi.

The paper is organized as follows. Section 2 recalls Lie--Rinehart and Gerstenhaber algebras. Section 3 gives the construction of the Schouten--Nijenhuis bracket, including a detailed verification of $A$-balancing, and proves antisymmetry and the Leibniz rule. Section 4 proves the graded Jacobi identity. Section 5 studies the inner graded derivations. Section 6 reconstructs the Lie--Rinehart structure from a Gerstenhaber bracket and proves the correspondence.

\section{Preliminaries}
\subsection{Lie--Rinehart algebras}
\begin{definition}
A \emph{Lie--Rinehart algebra over $A$} is an $A$-module $G$ endowed with a $K$-Lie bracket $[\cdot,\cdot]$ and an $A$-linear Lie algebra morphism
\[
\rho:G\longrightarrow\Der_K(A)
\]
such that
\begin{equation}\label{LR}
[x,ay]=a[x,y]+\rho(x)(a)y
\end{equation}
for all $a\in A$ and $x,y\in G$.
\end{definition}

By skew-symmetry, \eqref{LR} is equivalent to
\begin{equation}\label{LRleft}
[ax,y]=a[x,y]-\rho(y)(a)x.
\end{equation}

\begin{example}
For a smooth manifold $M$, the triple
\[
A=C^\infty(M),\qquad G=\mathfrak X(M),\qquad \rho=\mathrm{id}
\]
is a Lie--Rinehart algebra, since
\[
[X,fY]=f[X,Y]+X(f)Y.
\]
\end{example}

\begin{example}
For every commutative unital $K$-algebra $A$, $(\Der_K(A),\mathrm{id})$ is a Lie--Rinehart algebra over $A$.
\end{example}

\subsection{Gerstenhaber algebras}
Let
\[
\mathcal G=\bigwedge_A G=\bigoplus_{p\ge0}\mathcal G^p,
\qquad \mathcal G^p=\bigwedge_A^pG,
\]
with $\mathcal G^0=A$, $\mathcal G^1=G$, and $\mathcal G^p=0$ for $p<0$.

\begin{definition}[\cite{Gerstenhaber1963}]
A \emph{Gerstenhaber algebra} is a graded commutative associative algebra $(\mathcal G,\wedge)$ endowed with a $K$-bilinear bracket of degree $-1$ such that, for homogeneous $P\in\mathcal G^p$, $Q\in\mathcal G^q$ and $R\in\mathcal G^r$,
\begin{align}
[P,Q]&=-(-1)^{(p-1)(q-1)}[Q,P],\label{GAanti}\\
(-1)^{(p-1)(r-1)}[P,[Q,R]]&+(-1)^{(q-1)(p-1)}[Q,[R,P]]\notag\\
&+(-1)^{(r-1)(q-1)}[R,[P,Q]]=0,\label{GAjac}\\
[P,Q\wedge R]&=[P,Q]\wedge R+(-1)^{(p-1)q}Q\wedge[P,R].\label{GAleib}
\end{align}
\end{definition}

\subsection{Graded derivations}
\begin{definition}
A \emph{graded derivation of degree $d$} of $\mathcal G$ is a $K$-linear map $D:\mathcal G\to\mathcal G$ with $D(\mathcal G^k)\subseteq\mathcal G^{k+d}$ and
\[
D(P\wedge Q)=D(P)\wedge Q+(-1)^{dp}P\wedge D(Q),\qquad P\in\mathcal G^p .
\]
We denote by $\Der_K(\mathcal G)$ the graded space of graded derivations. The graded commutator
\[
[D_1,D_2]=D_1\circ D_2-(-1)^{d_1d_2}D_2\circ D_1
\]
of derivations of degrees $d_1,d_2$ is a graded derivation of degree $d_1+d_2$, and $\Der_K(\mathcal G)$ is a graded Lie algebra.
\end{definition}

\begin{lemma}\label{lem:dergens}
A graded derivation of $\mathcal G=\bigwedge_A G$ is uniquely determined by its values on $A$ and $G$.
\end{lemma}
\begin{proof}
As a $K$-algebra, $\mathcal G$ is generated by $A\cup G$: every element is a finite sum of products $a\,x_1\wedge\cdots\wedge x_p$. The graded Leibniz rule determines $D$ on such products from its values on $a,x_1,\dots,x_p$, and $K$-linearity gives the result.
\end{proof}

\section{Construction of the Schouten--Nijenhuis bracket}
\subsection{Defining formulas}
For $x_1,\dots,x_p\in G$ we write
\[
X=x_1\wedge\cdots\wedge x_p,\qquad
\widehat X_i=x_1\wedge\cdots\widehat{x_i}\cdots\wedge x_p,
\qquad\text{so that}\qquad x_i\wedge\widehat X_i=(-1)^{i-1}X,
\]
and similarly $Y=y_1\wedge\cdots\wedge y_q$, $\widehat Y_j$.

For $p,q\ge1$ define
\begin{equation}\label{SNformula}
[X,Y]_{SN}=\sum_{i=1}^p\sum_{j=1}^q(-1)^{i+j}[x_i,y_j]\wedge\widehat X_i\wedge\widehat Y_j .
\end{equation}
For $a\in A$ and $p\ge1$ define
\begin{equation}\label{SNfunction}
[X,a]_{SN}=\sum_{i=1}^p(-1)^{p-i}\rho(x_i)(a)\,\widehat X_i,
\qquad
[a,X]_{SN}=\sum_{i=1}^p(-1)^{i}\rho(x_i)(a)\,\widehat X_i=(-1)^p[X,a]_{SN},
\end{equation}
and finally $[a,b]_{SN}=0$ for $a,b\in A$. In low degrees these formulas give
\begin{equation}\label{gens}
[a,b]_{SN}=0,\qquad [x,a]_{SN}=\rho(x)(a),\qquad [a,x]_{SN}=-\rho(x)(a),\qquad [x,y]_{SN}=[x,y].
\end{equation}

\begin{remark}
The sign $(-1)^{p-i}$ in \eqref{SNfunction} is forced by graded antisymmetry: for $p=2$ one must have $[u\wedge v,a]_{SN}=[a,u\wedge v]_{SN}=\rho(v)(a)u-\rho(u)(a)v$. It coincides with $(-1)^{i+1}$ only when $p$ is odd.
\end{remark}

\subsection{Compatibility with the $A$-balanced exterior algebra}
The space $\bigwedge_A^pG$ is the quotient of the free $K$-module on $G^p$ by the relations expressing $K$-multilinearity, $A$-balancing
\[
(\dots,ax_k,\dots,x_l,\dots)\sim(\dots,x_k,\dots,ax_l,\dots),
\]
and alternation (vanishing when two entries coincide). Hence a map defined on tuples descends to $\bigwedge_A^pG$ as soon as it is $K$-multilinear, $A$-balanced and alternating. Since the Lie bracket of $G$ is not $A$-bilinear, $A$-balancing is the non-trivial point. We first treat the two smallest cases.

\begin{lemma}\label{lem:degree21}
For $a\in A$ and $x,z,y\in G$,
\[
[(ax)\wedge z,y]_{SN}=[x\wedge(az),y]_{SN}.
\]
\end{lemma}
\begin{proof}
In degree $(2,1)$ formula \eqref{SNformula} gives
\[
[u\wedge v,y]_{SN}=[u,y]\wedge v+u\wedge[v,y].
\]
Therefore, using \eqref{LRleft},
\begin{align*}
[(ax)\wedge z,y]_{SN}
&=[ax,y]\wedge z+ax\wedge[z,y]
=a[x,y]\wedge z-\rho(y)(a)x\wedge z+a x\wedge[z,y],\\
[x\wedge(az),y]_{SN}
&=[x,y]\wedge az+x\wedge[az,y]
=a[x,y]\wedge z+a x\wedge[z,y]-\rho(y)(a)x\wedge z.
\end{align*}
The two expressions coincide.
\end{proof}

\begin{lemma}\label{lem:degree12}
For $a\in A$ and $x,y_1,y_2\in G$,
\[
[x,(ay_1)\wedge y_2]_{SN}=[x,y_1\wedge(ay_2)]_{SN}.
\]
\end{lemma}
\begin{proof}
In degree $(1,2)$ formula \eqref{SNformula} gives $[x,u\wedge v]_{SN}=[x,u]\wedge v+u\wedge[x,v]$. Thus, by \eqref{LR},
\begin{align*}
[x,(ay_1)\wedge y_2]_{SN}
&=a[x,y_1]\wedge y_2+\rho(x)(a)y_1\wedge y_2+a y_1\wedge[x,y_2],\\
[x,y_1\wedge(ay_2)]_{SN}
&=a[x,y_1]\wedge y_2+a y_1\wedge[x,y_2]+\rho(x)(a)y_1\wedge y_2.
\end{align*}
Again the expressions are identical.
\end{proof}

\begin{proposition}[General balancing]\label{prop:balancing}
Formulas \eqref{SNformula} and \eqref{SNfunction} descend to well-defined $K$-bilinear maps on $\bigwedge_A G$.
\end{proposition}
\begin{proof}
Denote by $F(x_1,\dots,x_p;y_1,\dots,y_q)$ the right-hand side of \eqref{SNformula}, an element of $\bigwedge_AG$. It is clearly $K$-multilinear.

\emph{Antisymmetry at the level of tuples.} Exchanging the roles of the two tuples, the term indexed by $(j,i)$ in $F(y;x)$ is
\[
(-1)^{i+j}[y_j,x_i]\wedge\widehat Y_j\wedge\widehat X_i
=-(-1)^{(p-1)(q-1)}(-1)^{i+j}[x_i,y_j]\wedge\widehat X_i\wedge\widehat Y_j,
\]
so that
\begin{equation}\label{tupleanti}
F(y;x)=-(-1)^{(p-1)(q-1)}F(x;y).
\end{equation}
Consequently it suffices to prove that $F$ is alternating and $A$-balanced in the first group of variables.

\emph{Alternation.} Suppose $x_k=x_l$ with $k<l$. Every term with $i\notin\{k,l\}$ contains $x_k\wedge x_l=0$. For the two remaining terms, $\widehat X_l=(-1)^{l-k-1}\widehat X_k$, whence
\[
(-1)^{k+j}[x_k,y_j]\wedge\widehat X_k\wedge\widehat Y_j+(-1)^{l+j}[x_l,y_j]\wedge\widehat X_l\wedge\widehat Y_j
=(-1)^{j}\bigl((-1)^k-(-1)^{k}\bigr)[x_k,y_j]\wedge\widehat X_k\wedge\widehat Y_j=0.
\]

\emph{$A$-balancing.} Fix $k<l$, $a\in A$, and compare the tuple $x^{(k)}$ in which $x_k$ is replaced by $ax_k$ with the tuple $x^{(l)}$ in which $x_l$ is replaced by $ax_l$. Terms with $i\notin\{k,l\}$ coincide by $A$-multilinearity of the exterior product. For fixed $j$, the difference of the terms with $i=k$ is, by \eqref{LRleft},
\[
(-1)^{k+j}\bigl([ax_k,y_j]-a[x_k,y_j]\bigr)\wedge\widehat X_k\wedge\widehat Y_j
=-(-1)^{k+j}\rho(y_j)(a)\,x_k\wedge\widehat X_k\wedge\widehat Y_j
=(-1)^{j}\rho(y_j)(a)\,X\wedge\widehat Y_j,
\]
and the difference of the terms with $i=l$ is
\[
-(-1)^{l+j}\bigl([ax_l,y_j]-a[x_l,y_j]\bigr)\wedge\widehat X_l\wedge\widehat Y_j
=(-1)^{l+j}\rho(y_j)(a)\,x_l\wedge\widehat X_l\wedge\widehat Y_j
=-(-1)^{j}\rho(y_j)(a)\,X\wedge\widehat Y_j .
\]
These two contributions cancel, hence $F(x^{(k)};y)=F(x^{(l)};y)$. By \eqref{tupleanti} the same holds in the second group of variables. Lemmas \ref{lem:degree21} and \ref{lem:degree12} are the cases $(p,q)=(2,1)$ and $(1,2)$.

\emph{Formula \eqref{SNfunction}.} It is $K$-multilinear. It is alternating by the same computation as above, since $(-1)^{p-k}+(-1)^{p-l}(-1)^{l-k-1}=0$. It is $A$-balanced because the anchor is $A$-linear: moving $a_0\in A$ from $x_k$ to $x_l$ does not change the terms $i\notin\{k,l\}$, and the terms $i=k$ (resp.\ $i=l$) agree because $\rho(a_0x)(a)=a_0\rho(x)(a)$.
\end{proof}

\begin{remark}
The proof makes the role of the Lie--Rinehart identity transparent. The defects
\[
[x,ay]-a[x,y]=\rho(x)(a)y,
\qquad
[ax,y]-a[x,y]=-\rho(y)(a)x
\]
produce two correction terms which cancel each other in the $A$-balanced exterior algebra.
\end{remark}

\subsection{Antisymmetry and Leibniz rules}
\begin{proposition}[Graded antisymmetry]\label{prop:anti}
For homogeneous $P\in\mathcal G^p$ and $Q\in\mathcal G^q$,
\[
[P,Q]_{SN}=-(-1)^{(p-1)(q-1)}[Q,P]_{SN}.
\]
\end{proposition}
\begin{proof}
For $p,q\ge1$ this is \eqref{tupleanti}. For $q=0$ and $p\ge1$, the required identity reads $[P,a]_{SN}=(-1)^p[a,P]_{SN}$, which is \eqref{SNfunction}. The case $p=q=0$ is trivial.
\end{proof}

\begin{proposition}[Leibniz rule]\label{prop:leibniz}
For homogeneous $P\in\mathcal G^p$, $Q\in\mathcal G^q$, $R\in\mathcal G^r$,
\begin{equation}\label{Leib2}
[P,Q\wedge R]_{SN}=[P,Q]_{SN}\wedge R+(-1)^{(p-1)q}Q\wedge[P,R]_{SN}.
\end{equation}
Consequently,
\begin{equation}\label{Leib1}
[P\wedge Q,R]_{SN}=P\wedge[Q,R]_{SN}+(-1)^{q(r-1)}[P,R]_{SN}\wedge Q.
\end{equation}
\end{proposition}
\begin{proof}
By bilinearity we may assume $P,Q,R$ decomposable.

\emph{Case $p,q,r\ge1$.} Write $Q\wedge R=y_1\wedge\cdots\wedge y_q\wedge z_1\wedge\cdots\wedge z_r$. In \eqref{SNformula}, the terms in which $x_i$ is bracketed with some $y_j$ give $[P,Q]_{SN}\wedge R$. The terms in which $x_i$ is bracketed with $z_k$ (in position $q+k$) are
\[
(-1)^{i+q+k}[x_i,z_k]\wedge\widehat X_i\wedge Q\wedge\widehat Z_k
=(-1)^{q+pq}\,Q\wedge\bigl((-1)^{i+k}[x_i,z_k]\wedge\widehat X_i\wedge\widehat Z_k\bigr),
\]
and $(-1)^{q+pq}=(-1)^{(p-1)q}$.

\emph{Case $p\ge1$, $Q=b\in A$, $r\ge1$.} Write $bR=(bz_1)\wedge z_2\wedge\cdots\wedge z_r$. In \eqref{SNformula}, the terms with $k\ge2$ equal $b$ times the corresponding terms of $[P,R]_{SN}$. For $k=1$, \eqref{LR} gives $[x_i,bz_1]=b[x_i,z_1]+\rho(x_i)(b)z_1$. Hence
\[
[P,bR]_{SN}=b[P,R]_{SN}+\sum_{i=1}^p(-1)^{i+1}\rho(x_i)(b)\,z_1\wedge\widehat X_i\wedge\widehat Z_1 .
\]
Moving $z_1$ past the $p-1$ factors of $\widehat X_i$, the sum becomes $\sum_i(-1)^{p-i}\rho(x_i)(b)\widehat X_i\wedge R=[P,b]_{SN}\wedge R$ by \eqref{SNfunction}. This is \eqref{Leib2}.

\emph{Case $p\ge1$, $R=b\in A$.} Then $Q\wedge b=bQ$ and, by the previous case, $[P,bQ]_{SN}=b[P,Q]_{SN}+[P,b]_{SN}\wedge Q$. Since $[P,b]_{SN}$ has degree $p-1$, $[P,b]_{SN}\wedge Q=(-1)^{(p-1)q}Q\wedge[P,b]_{SN}$, which is \eqref{Leib2}. If $Q,R\in A$, \eqref{Leib2} states that each $\rho(x_i)$ is a derivation of $A$.

\emph{Case $P=a\in A$.} If $q,r\ge1$, formula \eqref{SNfunction} gives directly
\[
[a,Q\wedge R]_{SN}=\sum_{j}(-1)^j\rho(y_j)(a)\widehat Y_j\wedge R+\sum_k(-1)^{q+k}\rho(z_k)(a)Q\wedge\widehat Z_k
=[a,Q]_{SN}\wedge R+(-1)^{q}Q\wedge[a,R]_{SN},
\]
and $(-1)^q=(-1)^{(p-1)q}$ for $p=0$. If $Q=b\in A$, then $[a,bR]_{SN}=b[a,R]_{SN}$ by $A$-linearity of $\rho$, which is \eqref{Leib2} since $[a,b]_{SN}=0$. The case $Q,R\in A$ is trivial.

Finally, \eqref{Leib1} follows from \eqref{Leib2} by applying Proposition \ref{prop:anti} three times.
\end{proof}

\begin{corollary}\label{cor:xder}
For $x\in G$,
\[
[x,x_1\wedge\cdots\wedge x_p]_{SN}
=\sum_{i=1}^p x_1\wedge\cdots\wedge[x,x_i]\wedge\cdots\wedge x_p.
\]
In particular, no alternating factor $(-1)^i$ occurs in this formula.
\end{corollary}

\begin{corollary}\label{cor:ader}
For $a\in A$,
\[
[a,x_1\wedge\cdots\wedge x_p]_{SN}
=\sum_{i=1}^p(-1)^i\rho(x_i)(a)\,
 x_1\wedge\cdots\widehat{x_i}\cdots\wedge x_p .
\]
\end{corollary}

\section{The graded Jacobi identity}
By Proposition \ref{prop:leibniz}, for homogeneous $P\in\mathcal G^p$ the operator
\[
\adSN(P):\mathcal G\longrightarrow\mathcal G,\qquad \adSN(P)(Q)=[P,Q]_{SN},
\]
is a graded derivation of degree $p-1$. We stress that this fact relies only on the Leibniz rule; no form of the Jacobi identity is involved.

For homogeneous $P,Q,R$, define the Jacobiator
\begin{align*}
J(P,Q,R)={}&(-1)^{(p-1)(r-1)}[P,[Q,R]_{SN}]_{SN}
+(-1)^{(q-1)(p-1)}[Q,[R,P]_{SN}]_{SN}\\
&+(-1)^{(r-1)(q-1)}[R,[P,Q]_{SN}]_{SN},
\end{align*}
and its Leibniz form
\[
J'(P,Q,R)=[P,[Q,R]_{SN}]_{SN}-[[P,Q]_{SN},R]_{SN}-(-1)^{(p-1)(q-1)}[Q,[P,R]_{SN}]_{SN}.
\]

\begin{lemma}\label{lem:Jforms}
For homogeneous $P,Q,R$:
\begin{enumerate}[label=\textup{(\roman*)}]
\item $J(P,Q,R)=J(Q,R,P)$;
\item $J'(P,Q,R)=(-1)^{(p-1)(r-1)}J(P,Q,R)$;
\item $J'(P,Q,\cdot)=[\adSN(P),\adSN(Q)]-\adSN([P,Q]_{SN})$ is a graded derivation of degree $p+q-2$.
\end{enumerate}
\end{lemma}
\begin{proof}
(i) is immediate from the definition. For (ii), multiply $J(P,Q,R)$ by $(-1)^{(p-1)(r-1)}$ and use Proposition \ref{prop:anti} in the form $[R,P]_{SN}=-(-1)^{(r-1)(p-1)}[P,R]_{SN}$ and $[R,[P,Q]_{SN}]_{SN}=-(-1)^{(r-1)(p+q-2)}[[P,Q]_{SN},R]_{SN}$; the signs combine to those of $J'$. For (iii), the identity is the definition of the graded commutator; $[\adSN(P),\adSN(Q)]$ and $\adSN([P,Q]_{SN})$ are graded derivations of degree $p+q-2$, hence so is their difference.
\end{proof}

\begin{lemma}\label{lem:JacGen}
The Jacobiator vanishes on all triples of elements of $A\cup G$.
\end{lemma}
\begin{proof}
For $x,y,z\in G$, this is the Jacobi identity in $G$. For $x,y\in G$ and $a\in A$,
\[
J(x,y,a)=\rho(x)\rho(y)(a)-\rho(y)\rho(x)(a)-\rho([x,y])(a)=0
\]
because $\rho$ is a Lie algebra morphism; by Lemma \ref{lem:Jforms}(i) this covers $J(y,a,x)$ and $J(a,x,y)$, hence all triples with two entries in $G$ and one in $A$. If at least two entries belong to $A$, every term vanishes: each term contains either a bracket of two elements of $A$, or a bracket of an element of $A$ with an element of $[G,A]_{SN}\subseteq A$.
\end{proof}

\begin{theorem}[Graded Jacobi identity]\label{thm:Jacobi}
For all homogeneous $P,Q,R\in\mathcal G$, $J(P,Q,R)=0$.
\end{theorem}
\begin{proof}
By Lemma \ref{lem:Jforms}, up to signs, $J(P,Q,R)$ equals $J'(P,Q,R)$, $J'(Q,R,P)$ and $J'(R,P,Q)$, and each of these is a graded derivation in its last argument. Using Lemma \ref{lem:dergens} three times:
\begin{enumerate}[label=(\alph*)]
\item $J(P,Q,\cdot)=0$ as soon as $J(P,Q,z)=0$ for all $z\in A\cup G$;
\item $J(P,Q,z)=\pm J'(z,P,Q)$ vanishes for all $Q$ as soon as $J(P,w,z)=0$ for all $w\in A\cup G$;
\item $J(P,w,z)=\pm J'(w,z,P)$ vanishes for all $P$ as soon as $J(u,w,z)=0$ for all $u\in A\cup G$.
\end{enumerate}
The last condition holds by Lemma \ref{lem:JacGen}.
\end{proof}

\begin{theorem}\label{thm:Gerstenhaber}
The triple
\[
\left(\bigwedge_A G,\wedge,[\cdot,\cdot]_{SN}\right)
\]
is a Gerstenhaber algebra.
\end{theorem}
\begin{proof}
The exterior product is graded commutative and associative, and the bracket has degree $-1$. Proposition \ref{prop:anti}, Proposition \ref{prop:leibniz}, and Theorem \ref{thm:Jacobi} give respectively graded antisymmetry, the graded Leibniz rule, and the graded Jacobi identity.
\end{proof}

\section{Inner graded derivations}
The graded derivations $\adSN(P)$, $P\in\mathcal G$, are called the \emph{inner} (or \emph{adjoint}) graded derivations of the Gerstenhaber algebra $\mathcal G$.

\begin{proposition}
For $a,b\in A$ and $x,y\in G$,
\[
\adSN(a)(b)=0,\quad \adSN(a)(x)=-\rho(x)(a),\quad
\adSN(x)(a)=\rho(x)(a),\quad \adSN(x)(y)=[x,y].
\]
By Lemma \ref{lem:dergens}, these values determine $\adSN(a)$ and $\adSN(x)$; their actions on decomposable multivectors are given by Corollaries \ref{cor:ader} and \ref{cor:xder}.
\end{proposition}

\begin{theorem}[Adjoint commutator identity]\label{thm:ad}
For homogeneous $P\in\mathcal G^p$ and $Q\in\mathcal G^q$,
\[
[\adSN(P),\adSN(Q)]=\adSN([P,Q]_{SN}),
\]
where the bracket on the left is the graded commutator.
\end{theorem}
\begin{proof}
By Lemma \ref{lem:Jforms}(iii), $[\adSN(P),\adSN(Q)]-\adSN([P,Q]_{SN})=J'(P,Q,\cdot)$, which vanishes by Lemma \ref{lem:Jforms}(ii) and Theorem \ref{thm:Jacobi}.
\end{proof}

\begin{corollary}
The map
\[
\adSN:\mathcal G[1]\longrightarrow\Der_K(\mathcal G),
\]
where $\mathcal G[1]^k=\mathcal G^{k+1}$, is a morphism of graded Lie algebras.
\end{corollary}

\begin{remark}
The derivation property of $\adSN(P)$ is used in the proof of Theorem \ref{thm:Jacobi}, but it comes from the Leibniz rule alone. Theorem \ref{thm:ad} is a consequence of the graded Jacobi identity and is not used to prove it. This ordering avoids the circularity that may arise if the adjoint commutator identity and the graded Jacobi identity are derived from one another without an independent starting point.
\end{remark}

\section{Reconstruction of the Lie--Rinehart structure}
Let $G$ be an $A$-module and suppose that $\mathcal G=\bigwedge_A G$, with its exterior product, carries a Gerstenhaber bracket $[\cdot,\cdot]$ of degree $-1$. For degree reasons,
\[
[A,A]\subseteq\mathcal G^{-1}=0,\qquad [G,A]\subseteq A,\qquad [G,G]\subseteq G .
\]
Define
\begin{equation}\label{anchorback}
\rho(x)(a)=[x,a],\qquad [x,y]_G=[x,y],\qquad a\in A,\ x,y\in G.
\end{equation}
Recall that \eqref{Leib1} follows formally from \eqref{GAanti} and \eqref{GAleib}, so it holds for any Gerstenhaber bracket.

\begin{proposition}\label{prop:recLR}
$(G,[\cdot,\cdot]_G,\rho)$ is a Lie--Rinehart algebra over $A$.
\end{proposition}
\begin{proof}
\emph{Lie algebra.} For $p=q=r=1$, identities \eqref{GAanti} and \eqref{GAjac} are the skew-symmetry and Jacobi identity of $[\cdot,\cdot]_G$.

\emph{Derivations.} By \eqref{GAleib} with $p=1$, $q=r=0$: $[x,ab]=[x,a]b+a[x,b]$, so $\rho(x)\in\Der_K(A)$.

\emph{$A$-linearity of $\rho$.} By \eqref{Leib1} with $P=a$, $Q=x$, $R=b$:
\[
\rho(ax)(b)=[a\wedge x,b]=a\wedge[x,b]+(-1)^{1\cdot(-1)}[a,b]\wedge x=a\,\rho(x)(b).
\]

\emph{Identity \eqref{LR}.} By \eqref{GAleib} with $P=x$, $Q=a$, $R=y$:
$[x,ay]=[x,a]\,y+a[x,y]=\rho(x)(a)y+a[x,y]_G$.

\emph{Morphism.} For $p=q=1$, $r=0$, the Leibniz form of \eqref{GAjac} (Lemma \ref{lem:Jforms}(ii), whose proof uses only \eqref{GAanti}) gives
$[x,[y,a]]-[y,[x,a]]-[[x,y],a]=0$, that is
$\rho(x)\rho(y)(a)-\rho(y)\rho(x)(a)=\rho([x,y]_G)(a)$.
\end{proof}

\begin{lemma}\label{lem:uniqueness}
A $K$-bilinear bracket on $\mathcal G$ satisfying \eqref{GAanti} and \eqref{GAleib} is uniquely determined by its values on pairs of elements of $A\cup G$.
\end{lemma}
\begin{proof}
For fixed homogeneous $P$, the map $[P,\cdot]$ is a graded derivation by \eqref{GAleib}, hence is determined by the values $[P,z]$, $z\in A\cup G$, by Lemma \ref{lem:dergens}. By \eqref{GAanti}, $[P,z]=\pm[z,P]$, and $[z,\cdot]$ is again a graded derivation, determined by the values $[z,w]$, $w\in A\cup G$.
\end{proof}

\begin{theorem}[Lie--Rinehart--Gerstenhaber correspondence \cite{Huebschmann1990}]\label{thm:correspondence}
Let $A$ be a commutative unital $K$-algebra and $G$ an $A$-module. The constructions of Section 3 and of Proposition \ref{prop:recLR} are mutually inverse bijections between Lie--Rinehart structures on $G$ over $A$ and Gerstenhaber brackets of degree $-1$ on $\left(\bigwedge_A G,\wedge\right)$. Under this correspondence,
\[
\rho(x)(a)=[x,a]_{SN},\qquad [x,y]_{G}=[x,y]_{SN}.
\]
\end{theorem}
\begin{proof}
Starting from a Lie--Rinehart structure, the Schouten--Nijenhuis bracket of Theorem \ref{thm:Gerstenhaber} satisfies \eqref{gens}, so Proposition \ref{prop:recLR} returns the original anchor and bracket. Conversely, starting from a Gerstenhaber bracket $[\cdot,\cdot]$, let $[\cdot,\cdot]_{SN}$ be the Schouten--Nijenhuis bracket of the Lie--Rinehart structure \eqref{anchorback}. Both brackets vanish on $A\times A$, and they agree on $G\times A$, on $A\times G$ (by antisymmetry) and on $G\times G$. By Lemma \ref{lem:uniqueness}, they coincide.
\end{proof}

\section{Examples and comments}
\subsection{Vector fields}
For $A=C^\infty(M)$ and $G=\mathfrak X(M)$, since $TM$ is a vector bundle, $\bigwedge_{C^\infty(M)}\mathfrak X(M)$ identifies with the space of multivector fields, and the construction recovers the usual Schouten--Nijenhuis bracket (with the sign convention $[X,f]_{SN}=X(f)$; other conventions in the literature differ by a sign depending on the degrees). In particular, for a multivector field $P$, $\adSN(P)$ is the usual Schouten adjoint action.

\subsection{Derivations of a commutative algebra}
For $G=\Der_K(A)$ with the commutator bracket and identity anchor, $\bigwedge_A\Der_K(A)$ inherits the Gerstenhaber bracket described above. This purely algebraic example exhibits the same mechanism without referring to a smooth manifold.

\section{Conclusion}
We have presented a careful construction of the Gerstenhaber algebra associated with a Lie--Rinehart algebra. The main technical point is the passage from the Lie bracket on $G$ to a well-defined bracket on the $A$-balanced exterior algebra $\bigwedge_A G$. The Lie--Rinehart identities
\[
[x,ay]=a[x,y]+\rho(x)(a)y,
\qquad
[ax,y]=a[x,y]-\rho(y)(a)x
\]
produce correction terms which cancel in the exterior algebra, and this is what makes the passage possible.

Once the bracket is established independently, its graded antisymmetry and biderivation property, together with the Lie--Rinehart axioms on generators, yield the graded Jacobi identity and hence a Gerstenhaber algebra. The inner graded derivations
\[
\adSN(P)=[P,\cdot]_{SN}
\]
then satisfy
\[
[\adSN(P),\adSN(Q)]=\adSN([P,Q]_{SN}),
\]
which is recognized as the adjoint representation of the shifted graded Lie algebra.

The converse construction recovers both the Lie bracket and the anchor from a Gerstenhaber bracket, and the two constructions are mutually inverse. This formulation provides a convenient framework for related constructions in generalized Lie--Rinehart and Lie--Rinehart--Jacobi settings \cite{MandalMishra2018,Okassa2007,Okassa2008,Okassa2011}.

\section*{Funding}
This research received no external funding.

\section*{Author Contributions}
All authors contributed equally to this work. All authors read and approved the final manuscript.

\section*{Data Availability}
Not applicable.

\section*{Conflict of Interest}
The authors declare that they have no conflict of interest.

\section*{Code Availability}
Not applicable.

\end{document}